\documentclass{amsart}
\usepackage[ansinew]{inputenc}
\usepackage{amsmath,latexsym,amssymb,verbatim}
\usepackage[all]{xy}

\newtheorem{lema}{Lemma}[section]
\newtheorem{teo}[lema]{Theorem}
\newtheorem{pro}[lema]{Proposition}
\newtheorem{defi}[lema]{Definition}

\newtheorem{com}[lema]{Remark}
\newtheorem{eje}[lema]{Example}

\newtheorem{teo*}{Theorem}
\input arrow.tex

\begin{document}

\title{A characterization of linearly semisimple groups}

\author{Amelia \'{A}lvarez}
\address{Departamento de Matem\'{a}ticas, Universidad de Extremadura,
Avenida de Elvas s/n, 06071 Badajoz, Spain}
\email{aalarma@unex.es}

\author{Carlos Sancho}
\address{Departamento de Matem\'{a}ticas, Universidad de Salamanca,
Plaza de la Merced 1-4, 37008 Salamanca, Spain}
\email{mplu@usal.es}

\author{Pedro Sancho}
\address{Departamento de Matem\'{a}ticas, Universidad de Extremadura,
Avenida de Elvas s/n, 06071 Badajoz, Spain}
\email{sancho@unex.es}

\begin{abstract}
Let $G={\rm Spec}\, A$ be an affine $K$-group scheme and $\tilde A=\{ w \in A^* \colon \dim_K A^* \cdot w \cdot A^*< \infty \}$. Let $\left\langle -,-\right\rangle  \colon A^* \times \tilde{A} \to K$, $\left\langle w,\tilde{w} \right\rangle := tr(w \tilde{w} )$, be the trace form. We prove that $G$ is linearly reductive if and only if
the trace form is non-degenerate on $A^*$.
\end{abstract}

\maketitle

\section{Introduction}

Let $K$ be a field and let $\bar{K}$ be its algebraic closure. A finite $K$-algebra (associative with unit) $B$ is separable (i.e., $B \otimes_K \bar{K} = \prod_i M_{n_i}(\bar K)$) and $g.c.d.\{ n_i,{\rm char}\, K\}=1$, for every $i$, if and only if the trace form of $B$ is non-degenerate.% (see \cite{}).

Let $G$ be a finite group of order $n$ and let $KG$ be the enveloping algebra of $G$. $G$ is linearly semisimple if and only if $KG$ is a separable $K$-algebra. By Maschke's theorem $G$ is linearly semisimple if and only if $g.c.d.\{n,{\rm char}\, K\}=1$.
It holds that $G$ is linearly semisimple if and only if the trace form of $KG$ is non-degenerate. This theorem has been generalized to finite-dimensional Hopf algebras and Frobenius algebras (see \cite{Larson} and \cite{M}).

Now let $G={\rm Spec}\, A$ be an affine $K$-group scheme. Let $\tilde{A} = \{ w \in A^* \colon \dim_K A^* \cdot w \cdot A^*< \infty \} \subseteq A^*$, which is a bilateral ideal of $A^*$. Let $tr \colon \tilde{A} \to K$ defined by $tr(\tilde{w})=$ trace of the endomorphism of $\tilde{A}$, $\tilde{w'} \mapsto \tilde{w} \tilde{w'}$, and $\left\langle -,-\right\rangle  \colon A^* \times \tilde{A} \to K$, $\left\langle w,\tilde{w} \right\rangle := tr(w \tilde{w} )$, the associated trace form.

In this paper we prove that $G$ is linearly semisimple if and only if the trace form is non-degenerate on $A^*$, that is, $\varphi \colon A^* \to \tilde{A}^*$, $\varphi(w):= \left\langle w,- \right\rangle $ is injective. We also prove that $G$ is linearly semisimple if and only if the trace form is non-degenerate on $\tilde{A}$, that is, $\phi \colon \tilde{A} \to \tilde{A}^*$, $\phi(\tilde{w}) := \left\langle -, \tilde{w} \right\rangle $ is injective, and $\tilde{A}$ is dense in $A^*$.

$G={\rm Spec}\, A$ is linearly semisimple if and only if there exists $w_G \in {A^*}^G$ such that $w_G(1)=1$ (see \cite{Sw}). Let $* \colon A \to A$, $a\mapsto a^*$ be the morphism induced on $A$ by the morphism $G \to G$, $g\mapsto g^{-1}$ and let $F \colon A \to A^*$, $F(a):=w_G(a^*\cdot -)$ be the ``Fourier transform". In this situation we prove that $\phi \colon \tilde{A} \to \tilde{A}^*$ values bijectively in $A$ and that the inverse morphism $\phi^{-1} \colon A \to \tilde{A}$ coincides with the Fourier transform.

\section{Preliminary results}

If $G$ is a smooth algebraic group over an algebraically closed field  one may only consider the rational points of $G$ in order to resolve many questions in the theory of linear representations of $G$. In general, for any $K$-group, one may regard the functor of points of the group and its linear representations and in this way $K$-groups and their linear representations may be handled as mere sets, as it is well known.

Let $R$ be a commutative ring (associative with unit). All functors considered in this paper are covariant functors over the category of commutative $R$-algebras (with unit). Given an $R$-module $M$, the functor ${\mathcal M}$ defined by ${\mathcal M}(S) := M \otimes_R S$ is called a quasi-coherent $R$-module. The functors $M \rightsquigarrow {\mathcal M}$, ${\mathcal M} \rightsquigarrow {\mathcal M}(R)$ establish an equivalence between the category of $R$-modules and the category of quasi-coherent $R$-modules (\cite[1.12]{Amel}). In particular, ${\rm Hom}_R ({\mathcal M},{\mathcal M'}) = {\rm Hom}_R (M,M')$.

If $\mathbb M$ and $\mathbb N$ are functors of ${\mathcal R}$-modules, we will denote by ${\mathbb Hom}_{\mathcal R} (\mathbb M, \mathbb N)$ the functor of ${\mathcal R}$-modules $${\mathbb Hom}_{\mathcal R} (\mathbb M,\mathbb N)(S) := {\rm Hom}_{\mathcal S} (\mathbb M_{|S}, \mathbb N_{|S}),$$ where $\mathbb M_{| S}$ is the functor $\mathbb M$ restricted to the category of commutative $S$-algebras. We denote $\mathbb M^* = {\mathbb Hom}_{\mathcal R} (\mathbb M, {\mathcal R})$ and given an $R$-module $M$, then $${\mathcal M^*}(S)={\rm Hom}_S (M \otimes_R S,S) = {\rm Hom}_R (M,S).$$ We will say that ${\mathcal M}^*$ is an $R$-module scheme and it holds that ${\mathcal M}^{**} = {\mathcal M}$ (\cite[1.10]{Amel}).

\begin{com}
Every morphism $M^* \to M'^*$ considered in this paper will be a functorial morphism, that is, we will actually have a morphism ${\mathcal M}^* \to {\mathcal M}'^*$. Observe that the dual morphism is a morphism ${\mathcal M}' \to {\mathcal M}$, which is induced by a morphism $M'\to M$. Then $M^* \to M'^*$ is the dual morphism of $M' \to M$.
\end{com}

Let $G={\rm Spec}\, A$ be an affine $R$-group and let $G^{\cdot}$
be the functor of points of $G$, that is, $G^{\cdot}(S) = {\rm
Hom}_{R-alg} (A,S)$. ${\mathcal A}^*$ is an $R$-algebra scheme, that is, besides from being an $R$-module scheme it is a functor of algebras (${\mathcal A^*}(S)$ is a $S$-algebra such that $S \subseteq Z({\mathcal A^*}(S))$). One has a natural morphism of functors of monoids $G^{\cdot} \to {\mathcal A}^*$, because $G^{\cdot}(S) = {\rm Hom}_{R-alg} (A,S) \subset {\rm Hom}_R (A,S) = {\mathcal A}^*(S)$. If $\mathbb M^*$ is a functor of ${\mathcal R}$-modules (resp. ${\mathcal R}$-algebras), any morphism of functors of sets (resp. monoids) $G^{\cdot} \to \mathbb M^*$ factorizes uniquely through a morphism of functors of modules (resp. algebras) ${\mathcal A^*}\to \mathbb M^*$ (\cite[5.3]{Amel}). Therefore, the category of linear (and rational) representations of $G$ is equal to the category of quasi-coherent ${\mathcal A}^*$-modules (\cite[5.5]{Amel}).

\begin{com}
For abbreviation, we sometimes use $g \in G$ or $m \in \mathbb M$ to denote $g \in G^{\cdot}(S)$ or $m \in \mathbb M(S)$ respectively. Given $m \in \mathbb M(S)$ and a morphism of $R$-algebras $S \to S'$, we still denote by $m$ its image by the morphism $\mathbb M(S) \to \mathbb M(S')$.
\end{com}

\begin{defi}
Let $R = K$ be a field. We will say that an affine $K$-group scheme $G$ is linearly semisimple if it is linearly reductive, that is, if
every linear representation of $G$ is completely reducible.
\end{defi}

$G$ is linearly semisimple if and only if ${\mathcal  A}^*$ is semisimple, i.e., $A^* = \prod_i A_i^*$, where $A_i^*$ are simple (and finite)
$K$-algebras (\cite[6.8]{Amel}). If $K$ is an algebraically closed field, then $A_i^*={\rm End}_KV_i$, where $V_i$ are the simple $G$-modules (up to isomorphism), by Wedderburn's theorem.

Let  $V_0 = K$ be the trivial representation of $G$. The unit of $A_0^*={\rm End}_KV_0$, $w_G :=
1_0 \in A_0^*\subset A^*$ is said to be the ``invariant integral on $G$''. The invariant
integral on $G$ is characterized by being $G$-invariant and
normalized, that is, $w_G(1)=1$ (\cite[2.8]{Reynolds}).  
Let $\mathbb N=\mathbb M^*$ be a functor of $G$-modules (equivalently, a functor of ${\mathcal A}^*$-modules). The morphism $w_G \cdot: \mathbb N \to \mathbb N, n \mapsto w_G \cdot n$ is the unique projection of $G$-modules of $\mathbb N$ onto $\mathbb N^G$, and $\mathbb N^G = w_G \cdot \mathbb N$ (see \cite[2.3, 3.3]{Reynolds}).

\section{Trace form}

Let us assume that $R = K$ is a field.

Let ${\mathcal A}^*$ be a $\mathcal K$-algebra scheme and let $\tilde{A} = \{ w \in A^* \colon \dim_K A^* \cdot w < \infty, \, \dim_K w \cdot A^* < \infty\}$. It holds that $\tilde{A}$ is a bilateral ideal of $A^*$.

\begin{com}
Functorially, $\tilde{\mathcal A}$ is the maximal $\mathcal K$-quasi-coherent
bilateral ideal ${\mathcal I} \subset {\mathcal A}^*$: Given $w \in
\tilde{A}$, consider the morphism of ${\mathcal A}^*$-modules $f
\colon {\mathcal A}^* \to {\mathcal A}^*$, $f(w') := w' \cdot w$.
Let ${\mathcal B}^* = {\rm Im}\, f = {\mathcal A}^*/\ker f$. Observe
that $B^* = A^* \cdot w$, then $\dim_K B^* < \infty$ and ${\mathcal
B}^*$ is a $\mathcal K$-quasi-coherent module. Since $B^* \subset \tilde{A}$, then ${\mathcal B}^* \subset \tilde{\mathcal A}$ and $\mathcal
{\tilde{A}}$ is a left ${\mathcal A}^*$-submodule of ${\mathcal
A}^*$. Likewise, ${\mathcal{\tilde{A}}}$ is a right ${\mathcal
A}^*$-submodule of ${\mathcal A}^*$. Hence, ${\mathcal{\tilde{A}}}
\subseteq {\mathcal I}$. If ${\mathcal M}$ is a $\mathcal K$-quasi-coherent
left ${\mathcal A}^*$-module (resp. right ${\mathcal A}^*$-module)
and $m \in M$, then $\dim_K A^* \cdot m < \infty$ (resp. $\dim_K m
\cdot A^* < \infty$) by \cite[4.7]{Amel}. Now it follows easily that
${\mathcal I} = {\mathcal{\tilde{A}}}$.
\end{com}

\begin{defi} Let $\mathcal A^*$ be a $\mathcal K$-algebra scheme. Let $tr \colon \tilde{\mathcal A} \to \mathcal K$ be defined by $tr(\tilde{w}) :=$  trace of the endomorphism of $\tilde{A}$, $\tilde{w}' \mapsto \tilde{w} \tilde{w}'$. We will say that $$\left\langle -,- \right\rangle\colon \mathcal A^*\times \tilde{\mathcal A}\to \mathcal K,\,\, \left\langle w,\tilde{w} \right\rangle := tr(w \tilde{w})$$ is
the trace form of $\mathcal A^*$. \end{defi}

 We have the associated ``polarities'' to the trace form 
 $$  \varphi \colon \mathcal A^* \to \tilde{\mathcal A}^*,\,\varphi(w) := \left\langle w,- \right\rangle, \qquad\phi \colon \tilde{\mathcal A} \to \mathcal A^{**}=\mathcal A,\, \phi(\tilde{w}) := \left\langle -,\tilde{w}\right\rangle.$$

The dual morphism  $\varphi^* \colon \mathcal A^* \to \tilde{\mathcal A}^*$ of $\varphi$ is equal to  $\phi$:
$$\varphi^*(w)(\tilde w)=w(\varphi(\tilde w))=\phi(\tilde w)(w)=\langle w,\tilde w\rangle
=\varphi(w)(\tilde w)$$
Likewise, the dual morphism of $\phi$ is $\varphi$.

\begin{defi} Let $G={\rm Spec}\, A$ an affine $K$-group scheme. We will say that the
trace form of $\mathcal A^*$ is the trace form of $G$.\end{defi}

Let $C$ be a $K$-algebra ($K \subseteq Z(C)$). If $M$ is a left $C$-module (resp. a right $C$-module), then $M^*={\rm Hom}_K(M,K)$ is a right $C$-module (resp. a left $C$-module): $( w \cdot c)(m) := w(c \cdot m)$ (resp. $(c \cdot w)(m) := w(m \cdot c)$) for all $w \in M^*$, $m \in M$ and $c \in C$.

The morphism $\varphi\colon \mathcal A^*\to \tilde{\mathcal A}^*$ is a morphism of left and right $\mathcal A^*$-modules:
$$\aligned \varphi(w_1w_2)(\tilde w) & =tr(w_1w_2\tilde w)=tr(w_2\tilde ww_1)=\varphi(w_2)(\tilde ww_1)=(w_1\varphi(w_2))(\tilde w)\\
\varphi(w_1w_2)(\tilde w) & =tr(w_1w_2\tilde w)=\varphi(w_1)(w_2\tilde w)=(\varphi(w_1)w_2)(\tilde w)\endaligned$$
for all $w_1,w_2\in \mathcal A^*$ and $\tilde w\in\tilde{\mathcal A}$. Likewise, 
$\phi$ is a morphism of left and right $\mathcal A^*$-modules.

If $A^* = \prod_i A_i^*$, where $A_i^*$ are finite $K$-algebras, then $\tilde{A} = \oplus_i A_i^*$. In this case we have the morphism $\phi \colon \tilde{A} = \oplus A^*_i \to \oplus_i A_i=A$. Given $1_i := (0,\ldots,\overset{i}{1}, \ldots, 0) \in \tilde{A}$, then $\phi(1_i) = tr_{A_i^*}$. Obviously, $\tilde{A}$ is an $A^*$-module generated by $\left\{ 1_i \right\}_i$, and $$ \left\langle 1_i,1_j \right\rangle  = tr(1_i \cdot 1_j) = \left\{ \begin{array}{ll} \dim_K A^*_i & \text{ if }  i =j \\ 0 & \text{ otherwise } \end{array} \right.$$

\begin{eje}
Let $G_m= {\rm Spec}\, \mathbb{C}[x,1/x]$ be the multiplicative group. It holds that $\mathbb{C}[x,1/x]^* = \prod_{\mathbb{Z}} \mathbb{C}$, $\widetilde{\mathbb{C}[x,1/x]}= \oplus_{\mathbb{Z}} \mathbb{C}$ and $$\phi \colon \widetilde{\mathbb{C}[x,1/x]} = \underset{\mathbb{Z}}{\oplus} \mathbb{C} \to \underset{n \in \mathbb{Z}}{\oplus} \mathbb{C} \cdot x^n = \mathbb{C}[x,1/x],\quad \phi((a_n)_{n \in \mathbb{Z}}) = \sum_n a_nx^n.$$

Let $G_a = {\rm Spec}\, \mathbb{C}[x]$ be the additive group. It holds that $\mathbb{C}[x]^* = \mathbb{C}[[z]]$ ($z^n(x^m) = \delta_{nm} \cdot n!$) and $\widetilde{\mathbb{C}[[z]]} = 0$.
\end{eje}

Let $V$ be a linear representation of an affine $K$-group scheme $G={\rm Spec}\,
A$. The associated character $\chi_V \in A$ is defined by
$\chi_V(g)=$ trace of the linear endomorphism $V \to V$, $v
\mapsto g \cdot v$, for every $g \in G$ and $v \in V$.

Let $G={\rm Spec}\, A$ be a linearly semisimple affine $K$-group scheme. One has that $A^* = \prod_{i} A_i^*$, where $A_i^*$ are finite simple algebras. Assume for simplicity that $K$ is an algebraically closed field, then $A_i^* = {\rm End}_K (V_i)$. Observe that $tr_{A_i^*} = n_i \cdot \chi_{V_i}$, where $n_i = \dim_K V_i$.

Let  $V_0 = K$ be the trivial representation of $G$ and let $w_G :=
1_0 \in A^*$ be the ``invariant integral on $G$''.

Given $a \in A$, then $w_G \cdot a \in K = A^G$. Hence $w_G \cdot a = (w_G \cdot a)(1)= a(w_G) = w_G(a)$.

One has that $w_G \cdot A_j= 1_0 \cdot A_j=0$, if $j \neq 0$ and $w_G \cdot a_0 = a_0$ for all $a_0 \in A_0$. Hence, $w_G \cdot \chi_{V_j} = 0$ if $V_j$ is not the trivial representation, and $w_G \cdot \chi_{V_0} = \chi_{V_0} = 1$. Moreover, since $\chi_{V \oplus V'} = \chi_V + \chi_{V'}$, one has
\begin{equation}
w_G(\chi_V) = w_G \cdot \chi_V = \dim_K V^G .
\end{equation}

\section{Characterization of linearly semisimple groups}

Let $G={\rm Spec}\, A$ be an affine $K$-group scheme.

Let $* : A \to A$, $a \mapsto a^*$ be the morphism induced  by the
morphism $G \to G$, $g \mapsto g^{-1}$. In the following proof, if $V_i$ is a linear representation of $G$, we will consider $V_i^*$ as a left $G$-module by $(g*w)(v) = w(g^{-1}\cdot v)$. One has that $\chi_{V_i}^* = \chi_{V_i^*}$, because the trace of $g^{-1} \in G$ operating on $V_i$ is equal to the trace of $g$ operating on $V_i^*$ (which operates by the transposed inverse of $g$).

\begin{teo}\label{Parseval}
Let $G = {\rm Spec}\, A$  be a linearly semisimple affine $K$-group scheme and let $w_G \in A^*$ be its invariant integral. Let
$$F \colon A \to A^*,\, \, F(a) := w_G(a^* \cdot -)$$
be the ``Fourier Transform" (where $w_G (a^* \cdot -)(a') := w_G(a^* \cdot a')$). Then $F(A)= \tilde{A}$ and $F$ is equal to the inverse morphism of the polarity $\phi$ associated to the trace form of $G$.
\end{teo}

\begin{proof}
First let us prove that $F:{\mathcal A} \to {\mathcal A}^* ,\, F(a) := w_G(a^* \cdot -)$ is a morphism of left $G$-modules. For every point $g \in G$, $$\aligned w_G((g\cdot a)^*\cdot -)& \overset*=w_G((a^*\cdot g^{-1})\cdot -) \overset{**}=w_G(((a^*\cdot g^{-1})\cdot -)\cdot g)= w_G(a^*\cdot (-\cdot g)) \\ & = g\cdot (w_G(a^*\cdot -))\endaligned$$ where $\overset{*}{=}$ is due to $(g \cdot a)^*(g') = a({g'}^{-1} \cdot g) = a((g^{-1} \cdot g')^{-1}) = (a^* \cdot g^{-1}) (g')$, and $\overset{**}{=}$ is due to $g \cdot w_G = w_G$. Likewise, $F$ is a morphism of right $G$-modules. Then $F$ is a morphism of left and right $A^*$-modules and $F(A)\subseteq \tilde A$.

Assume $K$ is an algebraically closed field. Then $A^* = \prod_i A_i^*$, where $A_i^*= {\rm End}_K(V_i) = {\rm M}_{n_i}(K)$. Observe that $M_{n_i}(K) \simeq M_{n_i}(K)^*$ as left and right $M_{n_i}(K)$-modules. Then $\tilde{A} =\oplus_i M_{n_i}(K)^*$ and $A = \oplus_i M_{n_i}(K)$ are isomorphous as left and right $A^*$-modules. Let us see that $F \circ \phi= z\cdot $, where $z=(z_i) \in \prod_i K=Z(A^*)$:
$$F \circ \phi \in {\rm Hom}_{A^* \otimes A^*} (\tilde{A},\tilde{A}) = {\rm Hom}_{A^* \otimes A^*}(A,A) \subseteq {\rm Hom}_{A^* \otimes A^*} (A^*,A^*) = Z(A^*).$$

As $F(tr_{A_i^*}) = F(\phi(1_i)) = z \cdot 1_i =z_i\cdot 1_i$, then $F(\chi_{V_i}) = (z_i/n_i) \cdot 1_i$. Observe that
$$F(\chi_{V_i}) (\chi_{V_j}) = w_G (\chi_{V_i^* \otimes V_j}) = w_G \cdot \chi_{V_i^* \otimes V_j} = \dim_K {\rm Hom}_G (V_i,V_j) = \delta_{ij}.$$ Hence, $1 = F(\chi_{V_i})(\chi_{V_i}) = ((z_i/n_i) \cdot 1_i) (\chi_{V_i}) = (z_i/n_i) \cdot n_i = z_i$ and $F \circ \phi = {\rm Id}$. Then $\phi$ is an injective morphism. Since $\phi(A^*_i) \subseteq A_i$, then $\phi(A^*_i) = A_i$ and $\phi \colon \tilde{A} = \oplus_i A^*_i \to \oplus_i A_i = A$ is an isomorphism. As a consequence, $\phi \circ F={\rm Id}$.
\end{proof}

\begin{teo} Let 
$G={\rm Spec}\, A$ be an affine $K$-group scheme. $G$ is linearly semisimple  if and only if the polarity $\phi \colon \tilde{A} \to A$ associated to the trace form of $G$ is an isomorphism.
\end{teo}

\begin{proof}
Let us assume that $\phi \colon \tilde{A} \simeq A$ is a left and right $A^*$-module isomorphism. Observe that $1 \in A$ is left and right $G$-invariant. Let $w_G = \phi^{-1}(1)$. Given $w \in A^*$, $w \cdot w_G = w(1) w_G = w_G \cdot w$. Then it is easy to check that
$tr(w_G) = w_G(1)$ and in general $\left\langle w_G,w' \right\rangle = w'(1)w_G(1)$. As $\left\langle w_G,- \right\rangle = \phi(w_G)\neq 0$, therefore $w_G(1)\neq 0$. Then, $w'= w_G / w_G(1)$ is a normalized invariant integral on $G$ and $G$ is linearly semisimple.

Reciprocally, if $G$ is linearly semisimple, by the previous theorem $\phi$ is an isomorphism, whose inverse morphism is the Fourier transform.
\end{proof}

\begin{defi}
We will say that the trace form of $G={\rm Spec}\, A$ is non-degenerate on $A^*$ if 
the associated polarity, $\varphi \colon A^* \to \tilde{A}^*$, $\varphi(w) := \langle w,- \rangle$, is an injective morphism.
\end{defi}

\begin{teo} Let 
$G={\rm Spec}\, A$ be an affine $K$-group scheme. $G$ is  linearly semisimple if and only if its trace form is non-degenerate on $A^*$.
\end{teo}

\begin{proof}
If $\varphi \colon A^* \to \tilde{A}^*$ is injective, then
dually $\phi \colon \tilde{A} \to A$ is surjective. Let us denote by $i \colon \tilde{A} \to A^*$ the natural inclusion. From the commutative diagram

$$\xymatrix{ \tilde{A} \ar@{->>}[r]^-{\phi} \ar@{^{(}->}[d]^-i & A  \ar[d]^-{i^*} \\ A^* \ar@{^{(}->}[r]^-\varphi & \tilde{A}^*}$$
it follows that $\phi \colon \tilde{A} \to A$ is an isomorphism and we have that $G$ is linearly semisimple.

Reciprocally, if $G$ is linearly semisimple, by the previous theorem $\phi \colon \tilde{A} \to A$ is an isomorphism, then dually $\varphi \colon A^* \to \tilde{A}^*$ is an isomorphism.
\end{proof}

Given a vector space $V$, on $V^*$ we will consider the topology whose closed sets are $\{ V'^0 = (V/V')^*\}_{V' \subset V}$. Given a vector subspace $W \subseteq V^*$ we denote $W^0 = \{ v \in V \colon v(W) = 0 \}$. The closure of $W$ in $V^*$ coincides with
$W^{00}$. (Functorially, the closure of ${\mathcal W}$ in ${\mathcal V}^*$ is the minimum $K$-scheme of submodules of ${\mathcal V}^*$ which contains ${\mathcal W}$). The closure of $\tilde{A}$ in $A^*$ is a bilateral ideal.

\begin{defi}
We will say that the trace form of $G={\rm Spec}\, A$ is non-degenerate on $\tilde{A}$ if its associated polarity $\phi \colon \tilde{A} \to \tilde{A}^*$, $\phi(\tilde{w}) = \left\langle -,\tilde{w}\right\rangle $ is an injective morphism.
\end{defi}

\begin{teo}\label{Pontrjagin} Let $G={\rm Spec}\, A$ be an affine $K$-group scheme.
$G$ is a linearly semisimple group if and only if $\tilde{A}$ is dense in $A^*$ and its trace form is non-degenerate on $\tilde A$.
\end{teo}

\begin{proof}
Let us assume that the trace form of $G={\rm Spec}\, A$ is non-degenerate on $\tilde{A}$. Then, the morphism $\phi \colon \tilde{A} \to A$ is injective, hence dually the morphism $\varphi \colon A^* \to \tilde{A}^*$ is surjective. $\ker \varphi$ is a closed set of $A^*$ and $\ker \varphi \cap \tilde{A} = \ker \phi=0$. Taking closure we obtain that $\ker \varphi \cap A^* = 0$, therefore $\ker \varphi=0$.

In conclusion, $\varphi \colon A^* \to \tilde{A}^*$ is an isomorphism, then dually $\phi \colon \tilde{A} \to A$ is an isomorphism and we obtain that $G$ is linearly semisimple.
\end{proof}

\section{Pontrjagin's duality}

Now our aim is to write Theorem \ref{Pontrjagin} in terms of the classical Pontrjagin's duality, which establishes the duality between compact commutative groups and discrete commutative groups.

$\tilde A$ is dense in $A^*$ if and only if for every epimorphism of algebra schemes $\pi \colon {\mathcal A}^* \to {\mathcal B}^*$, where $\dim_K B^*<\infty$, it holds that $\pi(\tilde A)=B^*$. Since $B^*$ injects into ${\rm End}_K (B^*)$, we obtain the following proposition.

\begin{pro}
Let $G = {\rm Spec}\, A$ be an affine $K$-group scheme. $\tilde A$ is dense in $A^*$ if and only if for every linear representation, $G^\cdot \to {\mathbb End}_{\mathcal K} ({\mathcal V})$, where $\dim_K V <\infty$, the image of the induced morphism $A^* \to {\rm End}_K (V)$ coincides with the image of $\tilde A$.
\end{pro}

Let us assume that $\tilde A$ is dense in $A^*$. We want to determine the maximal bilateral ideals of $\tilde{A}$.

Let $V$ be a $G$-module. For every $v \in V$, it holds that $v \in \tilde A \cdot v$: the $G$-module $\langle v \rangle$ generated by $v$ is finite-dimensional, then there exists $w \in \tilde A$ such that the morphism $w \colon \langle v \rangle \to \langle v \rangle $, $v'\mapsto w \cdot v'$ is the identity morphism. Then $v = w \cdot v \in \tilde{A} \cdot v$.

In particular, for every $\tilde w \in \tilde A$ it holds that $\tilde w \in \tilde A \cdot \tilde w$.

Let $I \subseteq \tilde{A}$ be an ideal and $w \in A^*$, then $w \cdot I \subseteq I$: $w \cdot I \subseteq \tilde{A} \cdot (w \cdot I ) \subseteq \tilde{A} \cdot I \subseteq I$. Then, every bilateral ideal $I \subset \tilde{A}$ is a left and right $A^*$-submodule, then $\mathcal{I}$ is a left and right ${\mathcal A}^*$-module, by \cite[4.6]{Amel}. If $I \subsetneq \tilde{A}$ is a maximal bilateral ideal, then $\dim_K \tilde{A}/I<\infty$,
because if $\mathcal{M}$ is a left and right quasi-coherent $\mathcal{A}^*$-module and $m \in M$, then the left and right quasi-coherent $\mathcal{A}^*$-module generated by $m$ is finite dimensional, by \cite[4.7]{Amel}.

Let $I \subset \tilde A$ be a bilateral ideal of finite codimension
and $\tilde w \in \tilde A$ such that $\tilde w \cdot \colon \tilde
A/I \to \tilde A/I$ is the identity morphism. Then, $\bar{\tilde w}
\cdot \bar v = \bar v$, for every $\bar v \in \tilde A/I$, that is,
$\bar{\tilde w}$ is a left unit of $\tilde A/I$. Likewise, there
exists a right unit in $\tilde A/I$, then there exists a unit $1 \in
\tilde A/I$. The morphism $\pi \colon {\mathcal A}^* \to \mathcal{
\tilde{A}}/{\mathcal I}$, $\pi(w) = w \cdot 1$ is a surjective
morphism of $\mathcal K$-algebra schemes. If ${\mathcal J^*} = \ker \pi$, then $A^*/J^* = \tilde A/I$. Hence, the algebra scheme closure of ${\mathcal{\tilde{A}}}$ is ${\mathcal A}^*$: If ${\mathcal B}^*$ is an algebra scheme, then ${\mathcal B}^* = \underset{\underset{i}{\leftarrow}}{\lim} \, {\mathcal B}_i^*$ is an inverse limit of algebra schemes ${\mathcal B}_i^*$, where $\dim_K B_i<\infty$, by \cite[4.12]{Amel}. Then
$$\aligned {\rm Hom}_{\mathcal K-alg} (\mathcal{\tilde{A}},{\mathcal B}^*) &
= \underset{\underset{i}{\leftarrow}}{\lim}\, {\rm Hom}_{\mathcal K-alg}
(\mathcal{\tilde{A}},{\mathcal B}_i^*) =
\underset{\underset{i}{\leftarrow}}{\lim} \, {\rm
Hom}_{\mathcal K-alg}({\mathcal A}^*,{\mathcal B}_i^*)\\ & = {\rm
Hom}_{\mathcal K-alg}({\mathcal A}^*,{\mathcal B}^*). \endaligned$$

Let $G={\rm Spec}\, A$ be an affine $K$-group scheme and let $\hat G$ be the ``dual group'', that is, the set of irreducible representations of $G$, up to isomorphism. If $\tilde A$ is dense in $A^*$, then
$$\aligned {\rm Spec}\, \tilde A  & :=  \{\text{Maximal bilateral ideals $I \subsetneq \tilde A$} \} \\ & = \{ \text{Maximal bilateral ideal schemes $\mathcal{J}^* \subset \mathcal{A}^*$} \} \stackrel{\text{\cite[6.11-12]{Amel}}}{=} \hat G .\endaligned$$

\begin{defi}
We will say that $\hat G$ is discrete if $\tilde A$ is dense in $A^*$ and $\tilde A$ is separable (that is, $\tilde A$ is a direct sum of algebras of matrices by base change to the algebraic closure of $K$).
\end{defi}

Now we can rewrite Theorem \ref{Pontrjagin} as follows.

\begin{teo}
Let $G={\rm Spec}\, A$ be an affine $K$-group scheme. $G$ is linearly semisimple if and only if $\hat G$ is a discrete group.
\end{teo}

\begin{proof}
If $G$ is linearly semisimple, then $A^*= \prod_i A_i^*$, where $A_i^*$ is a finite simple algebra for all $i$. By base change $G$ is linearly semisimple, then $A_i^*$ is a direct sum of algebras of matrices by base change to the algebraic closure of $K$. Since
$\tilde{A} = \oplus_i A_i^*$, then $\tilde{A}$ is dense in $A^*$ and it is separable.

If $\tilde{A}$ is separable then its trace form is non-degenerate. Hence, if $\tilde{A}$ is separable and dense in $A^*$, then $G$ is linearly semisimple by Theorem \ref{Pontrjagin}.
\end{proof}


\begin{thebibliography}{99}




\bibitem[A1]{Amel} \textsc{\'{A}lvarez, A., Sancho, C., Sancho, P.,}
\textit{\!Algebra schemes and their representations}, J. Algebra
{\bf 296/1} (2006) 110-144.

\bibitem[A2]{Reynolds} \textsc{\'{A}lvarez, A., Sancho, C., Sancho, P.,} \textit{\!Reynolds operator}, arXiv:math/0611311v3 [math.AG] (2008).

\bibitem[LS]{Larson}\textsc{Larson, R.G., Sweedler, M.E.,} \textit{An associative orthogonal bilinear form for Hopf algebras}, Am. J. Math. {\bf 91} (1969) 75-93.

\bibitem[M]{M} \textsc{Murray, W.,} \textit{Bilinear forms on Frobenius algebras}, J. Algebra {\bf 293} (2005) 89-101.

\bibitem[S]{carlos} \textsc{Sancho de Salas, C.,} \!Grupos
algebraicos y teor\'{\i}a de invariantes, Sociedad Matem\'{a}tica
Mexicana, M\'{e}xico, 2001.

\bibitem[Sw]{Sw} \textsc{Sweedler, M.E.,} \textit{Hopf Algebras}, Mathematics Lecture Note Series 44, Addison-Wesley, New York, 1969.

\end{thebibliography}
\end{document}